\documentclass[11pt]{article}
\usepackage{amsfonts,epsf,amsmath,amssymb}
\usepackage{epsfig}
\usepackage{latexsym}
\usepackage{pgf,tikz}
\usepackage{bm}
\usepackage{lineno}

\newtheorem{theorem}{\bf Theorem}[section]

\newtheorem{lemma}[theorem]{\bf Lemma}
\newtheorem{proposition}[theorem]{\bf Proposition}

\newcommand{\qed}{\hfill $\square$ \bigskip}
\newcommand{\Fib}{{\cal F}}
\newcommand{\Luc}{{\cal L}}

\begin{document}
\title{The (non-)existence of perfect codes in Lucas cubes}

\author{
Michel Mollard\footnote{Institut Fourier, CNRS, Universit\'e Grenoble Alpes, France email: michel.mollard@univ-grenoble-alpes.fr}
}
\date{\today}
\maketitle

\begin{abstract}
\noindent The {\em Fibonacci cube} of dimension $n$, denoted as $\Gamma_n$,  is the subgraph of the $n$-cube $Q_n$ induced by vertices with no consecutive 1's. Ashrafi and his co-authors proved the non-existence of perfect codes in  $\Gamma_n$ for $n\geq 4$. As an open problem the authors suggest to consider the existence of perfect codes in generalizations of Fibonacci cubes. The most direct generalization is the family $\Gamma_n(1^s)$ of subgraphs induced by strings without $1^s$ as a substring where $s\geq 2$ is a given integer. In a precedent work we proved the existence of a perfect code in $\Gamma_n(1^s)$ for $n=2^p-1$ and $s \geq 3.2^{p-2}$ for any integer $p\geq 2$.\\
The Lucas cube $\Lambda_n$ is obtained from $\Gamma_n$ by removing vertices that start and end with 1. Very often the same problems are studied on Fibonacci cubes and Lucas cube. In this note we prove the non-existence of perfect codes in $\Lambda_n$ for $n\geq 4$ and prove the existence of perfect codes in some generalized Lucas cube $\Lambda_n(1^s)$.
\end{abstract}

\noindent
{\bf Keywords:} Error correcting codes, perfect code, Fibonacci cube. 

\noindent
{\bf AMS Subj. Class. }: 94B5,0C69

\section{Introduction and notations}

An interconnection topology can be represented by a graph $G=(V,E)$, where $V$ denotes the processors and $E$ the communication links.
The hypercube $Q_n$ is a popular interconnection network because of its structural properties.\\
\indent The Fibonacci cube was introduced in \cite{hsu93} as a new interconnection network.
This graph is an isometric subgraph of the hypercube which is inspired in the Fibonacci numbers. It has attractive recurrent structures such as
its decomposition into two subgraphs which are also Fibonacci cubes by themselves.  Structural properties of these graphs were more extensively 
studied afterwards. See \cite{survey} for a survey. \\
\indent Lucas cubes, introduced in \cite{mupe2001}, have attracted the attention as well due to the fact that these cubes are 
closely related to the Fibonacci cubes. They have also been widely studied
\cite{dedo2002, ra2013, ca2011, klmope2011, camo2012, km2012}. \\

\indent We will next define some concepts needed in this paper. 
Let $G$ be a connected graph. The \emph{open neighbourhood} of a vertex $u$ is $N_G(u)$ the set of vertices adjacent to $u$. The \emph{closed neighbourhood} of $u$ is $N_G[u]=N_G(u)\cup\{u\}$. The \emph{distance} between two vertices noted $d_G(x,y)$ is the length of a shortest path between $x$ and $y$. We have thus $N_G[u]=\{v \in V(G);d_G(u,v)\leq1\}$. We will use the notations $d(x,y)$ and $N[u]$ when the graph is unambiguous.\\
A \emph{dominating set} $D$ of $G$ is a set of vertices such that every vertex of $G$ belongs to the closed neighbourhood of at least one vertex of $D$. 
In \cite{Biggs}, Biggs initiated the study of perfect codes in graphs a generalization of classical 1-error perfect correcting codes. A \emph{code} $C$ in $G$ is a set of vertices $C$ such that for every pair of distinct vertices $c,c'$ of $C$ we have $N_G[c]\cap N_G[c']=\emptyset$ or equivalently such that $d_G(c,c')\geq3$.

A \emph{perfect code} of a graph $G$ is both a dominating set and a code. It is thus a set of vertices $C$  such that every vertex of $G$ belongs to the closed neighbourhood of exactly one vertex of $C$. A perfect code is also known as an efficient dominating set. The existence or non-existence of perfect codes have been considered for many graphs. See the introduction of \cite{aabfk2016} for some references.

The vertex set of the \emph{$n$-cube} $Q_n$ is the set $\mathbb{B}_n$ of  binary strings of length $n$, two vertices being adjacent if they differ in precisely one position. Classical 1-error correcting codes and perfect codes are codes and perfect codes in the graph $Q_n$. The \emph{weight} of a binary string is the number of 1's.
The concatenation of strings $\bm{x}$ and $\bm{y}$ is denoted $\bm{x}||\bm{y}$ or just $\bm{x}\bm{y}$ when there is no ambiguity. A string $\bm{f}$ is a \emph{substring} of a string $\bm{s}$ if there exist strings $\bm{x}$ and $\bm{y}$, may be empty, such that $\bm{s}=\bm{x}\bm{f}\bm{y}$.

A {\em Fibonacci string} of length $n$ is a binary string $\bm{b}=b_1\ldots b_n$ with $b_i\cdot b_{i+1}=0$ for $1\leq i<n$. In other words a Fibonacci string is a binary string without $11$ as substring.
The {\em Fibonacci cube} $\Gamma_n$ ($n\geq 1$) is the subgraph of $Q_n$ induced by the Fibonacci strings of length $n$.
Adjacent vertices in $\Gamma_n$ differ in one bit. Because of the empty string, $\Gamma_0 = K_1$. 

A Fibonacci string  of length $n$ is a \textit{Lucas string} if $b_1 \, \cdotp b_n \neq 1$. 
That is, a Lucas string has no two consecutive 1's including the first and the last elements of the string. 
The \textit{Lucas cube} $\Lambda_n$ is the subgraph of $Q_n$ induced by the Lucas strings of length $n$. 
We have $\Lambda_0 = \Lambda_1 = K_1$.\\
Let $\Fib_n$ and $\Luc_n$ be the set of strings of Fibonacci strings and Lucas strings of length $n$.\\
By  $\Gamma_{n,k}$ and $\Lambda_{n,k}$ we denote the vertices of of weight $k$ in respectively $\Gamma_n$ and $\Lambda_n$\\

Since $$\Luc_n=\{0\bm{s}; \bm{s}\in \Fib_{n-1}\}\cup\{10\bm{s}0; \bm{s}\in \Fib_{n-3}\} $$ and $$|\Gamma_{n,k}|=\binom{n-k+1}{k}$$
it is immediate  to derive  the following classical result.
\begin{proposition}\label{orderrr}
Let $n \ge 1$. The number of vertices of weight $k \leq n$ in $\Lambda_n$ is  
$$|\Lambda_{n,k}|= \binom{n-k}{k}+ \binom{n-k-1}{k-1}.$$
\end{proposition}
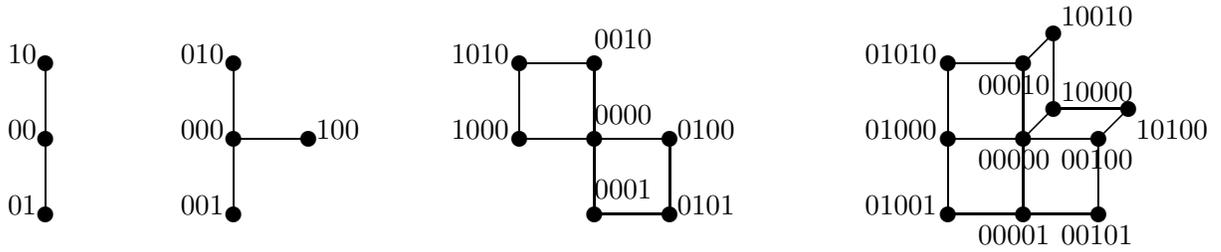
\begin{figure}
  \centering
\setlength{\unitlength}{1 mm}
\begin{picture}(160, 30)

\newsavebox{\gtwo}
\savebox{\gtwo}
  (30,30)[bl]
  {
  
\put(10,10){\circle*{2}}
\put(10,20){\circle*{2}}
\put(10,30){\circle*{2}}

\put(10,10){\line(0,1){10}}
\put(10,20){\line(0,1){10}}

\put(5,10){$01$}
\put(5,20){$00$}
\put(5,30){$10$}

  }

 \newsavebox{\lthree}
\savebox{\lthree}
  (40,30)[bl]
  {
  
\put(10,10){\circle*{2}}
\put(10,20){\circle*{2}}
\put(10,30){\circle*{2}}

\put(20,20){\circle*{2}}

\put(10,20){\line(1,0){10}}
\put(10,10){\line(0,1){10}}
\put(10,20){\line(0,1){10}} 

\put(3,10){$001$}
\put(3,20){$000$}
\put(3,30){$010$}
\put(21,20){$100$}

 }
 
\newsavebox{\lfour}
\savebox{\lfour}
  (40,30)[bl]
  {

\put(10,20){\circle*{2}}
\put(10,30){\circle*{2}}
\put(20,10){\circle*{2}}
\put(20,20){\circle*{2}}
\put(20,30){\circle*{2}}
\put(30,10){\circle*{2}}
\put(30,20){\circle*{2}}

\put(10,20){\line(1,0){10}}
\put(10,30){\line(1,0){10}}

\put(10,20){\line(0,1){10}}

\put(20,10){\line(1,0){10}}
\put(20,20){\line(1,0){10}}

\put(20,10){\line(0,1){10}}
\put(20,20){\line(0,1){10}}
\put(30,10){\line(0,1){10}} 

\put(20,12){$0001$}
\put(20,22){$0000$}
\put(20,32){$0010$}
\put(31,20){$0100$}
\put(1,20){$1000$}
\put(1,30){$1010$}
\put(31,10){$0101$}
 }
  \newsavebox{\gggfour}
\savebox{\gggfour}
  (40,40)[bl]
  {
 
\put(10,10){\circle*{2}}
\put(10,20){\circle*{2}}
\put(10,30){\circle*{2}}
\put(20,10){\circle*{2}}
\put(20,20){\circle*{2}}
\put(20,30){\circle*{2}}
\put(30,10){\circle*{2}}
\put(30,20){\circle*{2}}

\put(10,10){\line(1,0){10}}
\put(20,10){\line(1,0){10}}
\put(10,20){\line(1,0){10}}
\put(20,20){\line(1,0){10}}
\put(10,30){\line(1,0){10}}
\put(10,10){\line(0,1){10}}
\put(10,20){\line(0,1){10}}
\put(20,10){\line(0,1){10}}
\put(20,20){\line(0,1){10}}
\put(30,10){\line(0,1){10}}
  }
\newsavebox{\lfive}
\savebox{\lfive}
  (40,40)[bl]
  {
  
\put(0,0){\usebox{\gggfour}}

\put(24,24){\circle*{2}}
\put(24,34){\circle*{2}}
\put(34,24){\circle*{2}}
\put(24,24){\line(1,0){10}}
\put(24,24){\line(0,1){10}}

\put(20,20){\line(1,1){4}}
\put(20,30){\line(1,1){4}}
\put(30,20){\line(1,1){4}} 
\put(14,6){$00001$}
\put(14,16){$00000$}
\put(14,26){$00010$}
\put(25,16){$00100$}
\put(-1,10){$01001$}
\put(25,25){$10000$}
\put(-1,20){$01000$}
\put(-1,30){$01010$}
\put(25,06){$00101$}
\put(35,20){$10100$}
\put(25,35){$10010$}

 }
  
\put(0,-5){\usebox{\gtwo}}  
\put(25,-5){\usebox{\lthree}}  
\put(63,-5){\usebox{\lfour}}
\put(120,-5){\usebox{\lfive}}
\end{picture}
\caption{$\Gamma_2=\Lambda_2$,  $\Lambda_3$, $\Lambda_4$ and $\Lambda_5$}\label{F1}

\end{figure}

It will be convenient to consider the binary strings of length $n$ as vectors   of $\mathbb{F}^n$ the vector space of dimension $n$ over the field $\mathbb{F}=\mathbb{Z}_2$ thus to associate to a string $x_1 x_2 \dots x_n$ the vector $\theta(x_1 x_2 \dots x_n)=(x_1,x_2,\ldots,x_n)$.
The \emph{Hamming distance} between two vectors $\bm{x},\bm{y} \in \mathbb{F}^n$, $d(\bm{x},\bm{y})$ is the number of coordinates in which they differ.
By the correspondence $\theta$ we can define the binary sum $\bm{x}+\bm{y}$ and the Hamming distance $d(\bm{x},\bm{y})$ of  strings in $\mathbb{B}_n$. Note that the Hamming distance is the usual graph distance in $Q_n$.

We will first recall some basic results about perfect codes in $Q_n$.
Since $Q_n$ is a regular graph of degree $n$ the existence of a perfect code of cardinality $|C|$ implies $|C|(n+1)=2^n$ thus a necessary condition of existence is that $n+1$ is a power of 2 thus that $n=2^p-1$ for some integer $p$.

For any integer $p$ Hamming \cite{Ha1950} constructed, a linear subspace of $\mathbb{F}^{2^p-1}$ which is a perfect code. It is easy to prove that all linear perfect codes are Hamming codes. Notice that $1^n$ belongs to the Hamming code of length $n$. \\
In 1961  Vasilev \cite{{Va1962}}, and later many authors, see \cite{Co1,Sol2008} for a survey, constructed perfect codes which are not linear codes.\\

In a recent work  \cite{aabfk2016} Ashrafi and his co-authors proved the non-existence of perfect codes in  $\Gamma_n$ for $n\geq 4$. As an open problem the authors suggest to consider the existence of perfect codes in generalizations of Fibonacci cubes.
The most complete generalization  proposed in \cite{ikr} is, for a given string $\bm{f}$, to consider $\Gamma_n(\bm{f})$ the subgraph of $Q_n$ induced by strings that do not contain $\bm{f}$ as substring. Since Fibonacci cubes are $\Gamma_n(11)$ the most immediate generalization \cite{hsuliu,Zag} is to consider  $\Gamma_n({1^s})$ for a given integer $s$. In \cite {mo2018} we proved the existence of a perfect code in $\Gamma_n(1^s)$ for $n=2^p-1$ and $s \geq 3.2^{p-2}$ for any integer $p\geq 2$.\\

In the next section we will prove the main result of this note.

\begin{theorem}\label{thmain}
The Lucas cube $\Lambda_n , n\geq 0$, admits a perfect code if and only if $n\leq 3$ .
 \end{theorem}

\section{Perfect codes in Lucas cube}

It can be easily checked by hand that  $\{0^n\}$ is a perfect code of $\Lambda_n$ for $n\leq 3$ and that $\Lambda_4$ or $\Lambda_5$ does not contain a perfect code (Figure~\ref{F1}).\\
 Assume thus $n\geq 6$.\\ 

Note first that from Proposition \ref{orderrr} we have $$|\Lambda_{n,2}|=\frac{n(n-3)}{2} \text{ and }|\Lambda_{n,3}|=\frac{n(n-4)(n-5)}{6}$$.\\
Therefore $\Lambda_{n,2}$ and $\Lambda_{n,3}$ are none empty.\\
Let  $\Lambda_{n,k}^1$ be the vertices of $\Lambda_{n,k}$ that start with $1$. Since  $\Luc_n=\{0\bm{s}; \bm{s}\in \Fib_{n-1}\}\cup\{10\bm{s}0; \bm{s}\in \Fib_{n-3}\} $ the number of vertices in $\Lambda_{n,k}^1$ is $$|\Lambda_{n,k}^1|=|\Gamma_{n-3,k-1}|=\binom{n-1-k}{k-1}$$.

\begin{lemma}
If $n\geq 6$ and $C$ is a perfect code of $\Lambda_n$ then $0^n\in C$.
\end{lemma}
\begin{proof}

Suppose on the contrary that $0^n\notin C$. Since $0^n$ must be dominated there exists a vertex in $\Lambda_{n,1}\cap C$. This vertex is unique and because of the circular symmetry of $\Lambda_n$ we can assume $10^{n-1}\in C$. 

Since $0^n\notin C$ the other vertices of $\Lambda_{n,1}$ must be dominated by vertices in  $\Lambda_{n,2}$. But a vertex in  $\Lambda_{n,2}$ has precisely two neighbors in $\Lambda_{n,1}$ thus $n$ must be odd and $$|\Lambda_{n,2}\cap C|=\frac{n-1}{2}.$$\\
The unique vertex $10^{n-1}$ in $\Lambda_{n,1}\cap C$ has exactly  $n-3$ neighbors in  $\Lambda_{n,2}$. Let $D$ be the vertices of $\Lambda_{n,2}$ not in $C$ and not dominated by $10^{n-1}$. Vertices in  $D$ must be dominated by vertices in $\Lambda_{n,3}\cap C$. Each vertex of $\Lambda_{n,3}\cap C$ has exactly  exacty three neighbors  in  $\Lambda_{n,2}$. 
Thus $3$ divides the number of vertices in $D$. This number is  $$|D|=|\Lambda_{n,2}|-(n-3)-\frac{n-1}{2}=\frac{n^2-6n+7}{2}.$$
 This is not possible since there exists no odd integer $n$ such that $6$ divides $n^2+1$. Indeed since\textit{} $n$ is odd,  $6$ does not divide $n$ thus divides $(n+1)(n-1)= n^2-1$ or $(n+2)(n-2)=n^2-4$ or $(n+3)(n-3)=n^2-9$ thus cannot divide $n^2+1$.
\end{proof}
\qed  \\
Let $n\geq 6$ and $C$  be a perfect code. Since  $0^n\in C$   all vertices of  $\Lambda_{n,1}$ are dominated by  $0^n$ and thus $\Lambda_{n,2}\cap C=\Lambda_{n,1}\cap C=\emptyset$.
Consequently, each vertex of $\Lambda_{n,2}$ must be  dominated by a vertex in $\Lambda_{n,3}$.
Since each vertex in $\Lambda_{n,3}$ has precisely three neighbors in  $\Lambda_{n,2}$ we obtain that $$|\Lambda_{n,3}\cap C|=\frac{|\Lambda_{n,2}|}{3}.$$
This number must be an integer thus $3$ divides $|\Lambda_{n,2}|=\frac{n(n-3)}{2}$ and therefore $3$ divides $n(n-3)$. This is only possible if $n$ is a multiple of $3$.\\
Each vertex of $\Lambda_{n,2}^1$ must be  dominated by a vertex in $\Lambda_{n,3}^1$. Furthermore a vertex in  $\Lambda_{n,3}^1$  has precisely two neighbors in  $\Lambda_{n,2}^1$.  Therefore $|\Lambda_{n,2}^1|= n-3$ must be even and thus $n=6p+3$ for some integer $p\geq1$.\\
Let $E$ be  the  set of vertices of  $\Lambda_{n,3}$ not in $C$. Vertices in $E$ must be dominated by a vertex in $\Lambda_{n,4}$. Furthermore each vertex in $\Lambda_{n,4}$ has precisely four  neighbors in  $\Lambda_{n,3}$.\\ Therefore $4$ divides $|E|$ with $$|E|=|\Lambda_{n,3}|-|\Lambda_{n,3}\cap C|= \frac{n(n-4)(n-5)}{6}-\frac{n(n-3)}{6}=\frac{n(n^2-10n+23)}{6}. $$
Replacing $n$ by $6p+3$ we obtain that $4$ divides the odd number $(2p+1)(18p^2 -12p +1)$. This contradiction prove the Theorem.
\qed
 
\section{Perfect codes in generalized Lucas cube}
The analogous of the generalisation of Fibonacci cube $\Gamma_n(1^s)$ for Lucas cube  is the family $\Lambda_n(1^s)$ of subgraphs of $Q_n$ induced by strings without $1^s$ as a substring in a circular manner where $s\geq 2$ is a given integer.
More formally \cite{ikr2012} for any binary strings $b_1b_2\dots b_n$ and each $1 \leq i \leq n $, call $b_i b_{i+1}\dots b_n b_1 \dots b_{i-1}$ the $i$\emph{-th circulation} of $b_1b_2\dots b_n$. The \emph{generalized Lucas cube }$\Lambda_n(1^s)$ is the subgraph of $Q_n$ induced by strings without a circulation containing $1^s$ as a substring.\\
In \cite{mo2018} the existence of a perfect code in $\Gamma_n(1^s)$ is proved for $n=2^p-1$ and $s \geq 3.2^{p-2}$ for any integer $p\geq 2$.\\
The strategy used in this construction is to build a perfect code $C$ in $Q_n$ such that no vertex of $C$ contains $1^s$ as substring. The set $C$ is also a perfect code in $\Gamma_n(1^s)$ since each vertex of $\Gamma_n(1^s)$ belongs to the unique closed neighbourhood in $Q_n$ thus in $\Gamma_n(1^s)$ of a vertex in $C$. Because of the following proposition we cannot use the same idea  for $\Lambda_n(1^s)$ and $s\leq n-1$.
\begin{proposition}
Let $n$ an integer and $2\leq s\leq n-1$. There exist no perfect code $C$ in $Q_n$ such that the vertices of $C$ are without a circulation containing $1^s$ as a substring.
\end{proposition}
\begin {proof}

Let $C$ be a such a perfect code in $Q_n$ then $1^n\notin C$. Thus $1^n$ must be neighbour of a vertex $c$ in $C$. Since $c=1^i01^{n-1-i}$ for some integer $i$ the ${i+1}$th-circulation of $c$ is $1^{n-1}0$. 
\end {proof}
We can complete this proposition by the two following  results
\begin{proposition}
Let $p\geq2$ and $n=2^p-1$ then there exists a perfect code in $\Lambda_n(1^n)$ of order $|C|=\frac{2^n}{n+1}$.
\end{proposition}
\begin{proof} Let $D$ be a Hamming code of length $n$ and $C=\{d+(0^{n-1}1);d\in D\}$. Since $1^n\in D$ $C$ is a perfect code of $Q_n$ such that $1^n\notin C$. Since $\Lambda_n(1^n)$ is obtained from $Q_n$ by the deletion of $1^n$ every vertex of  $\Lambda_n(1^n)$ is in the closed neighbourhood of exactly one vertex of $C$.
\end{proof}
\begin{proposition}
Let $p\geq2$ and $n=2^p-1$ then there exists a perfect code in $\Lambda_n(1^{n-1})$ and in $\Lambda_n(1^{n-2})$ of order $|C|=\frac{2^n}{n+1}-1$.
\end{proposition}
\begin{proof} Let $D$ be a Hamming code of length $n$. Then $D$ is a perfect code of $Q_n$ such that $1^n\in D$. 
Since $\Lambda_n(1^{n-1})$ is obtained from $Q_n$ by the deletion of the closed neighbourhood of $1^n$ every vertex of  $\Lambda_n(1^{n-1})$ is in the closed neighbourhood of exactly one vertex of $C=D-\{1^n\}$.
 Furthermore since $1^n\in D$ there is no vertex of weight $n-2$ in $D$. Let $u$ be a vertex of  $\Lambda_n(1^{n-2})$ and $f(u)$ be the vertex in $D$ such that $u\in N_{Q^n}[u]$. Since there is no vertex in $D$ with weight $n-1$ or $n-2$ there is no circulation of $f(u)$ containing $1^{n-2}$ as a substring. 
Therefore $f(u)$ is a vertex of $\Lambda_n(1^{n-2})$ and $u\in N_{\Lambda_n(1^{n-2})}[f(u)]$.
Since a code in $Q_n$ is a code in each of its subgraph $C$ is a perfect code of $\Lambda_n(1^{n-2})$.
\end{proof}\textit{}
\bibliographystyle{plain}
\bibliography{PerfectcodesinLucascubes}

\begin{thebibliography}{10}

\bibitem{aabfk2016}
Ali~Reza Ashrafi, Jernej Azarija, Azam Babai, Khadijeh Fathalikhani, and Sandi
  Klav\v{z}ar.
\newblock The (non-)existence of perfect codes in fibonacci cubes.
\newblock {\em Information Processing Letters}, 116(5):387 -- 390, 2016.

\bibitem{Biggs}
Norman Biggs.
\newblock Perfect codes in graphs.
\newblock {\em Journal of Combinatorial Theory, Series B}, 15(3):289 -- 296,
  1973.

\bibitem{ca2011}
Aline Castro, Sandi Klav\v{z}ar, Michel Mollard, and Yoomi Rho.
\newblock On the domination number and the 2-packing number of fibonacci cubes
  and lucas cubes.
\newblock {\em Computers and Mathematics with Applications}, 61(9):2655 --
  2660, 2011.

\bibitem{camo2012}
Aline Castro and Michel Mollard.
\newblock The eccentricity sequences of fibonacci and lucas cubes.
\newblock {\em Discrete Mathematics}, 312(5):1025 -- 1037, 2012.

\bibitem{Co1}
G{\'e}rard Cohen, Iiro Honkala, Simon Litsyn, and Antoine Lobstein.
\newblock {\em Covering Codes Chapter 11}, volume~54 of {\em North-Holland
  Mathematical Library}.
\newblock Elsevier, 1997.

\bibitem{dedo2002}
Ernesto Ded{\'o}, Damiano Torri, and Norma~Zagaglia Salvi.
\newblock The observability of the fibonacci and the lucas cubes.
\newblock {\em Discrete Mathematics}, 255(1):55 -- 63, 2002.

\bibitem{Ha1950}
R.~W. {Hamming}.
\newblock Error detecting and error correcting codes.
\newblock {\em The Bell System Technical Journal}, 29(2):147--160, 1950.

\bibitem{hsu93}
W.-J. {Hsu}.
\newblock Fibonacci cubes-a new interconnection topology.
\newblock {\em IEEE Transactions on Parallel and Distributed Systems},
  4(1):3--12, 1993.

\bibitem{ikr}
Aleksandar Ili\'c, Sandi Klav\v{z}ar, and Yoomi Rho.
\newblock Generalized fibonacci cubes.
\newblock {\em Discrete Mathematics}, 312(1):2 -- 11, 2012.
\newblock Algebraic Graph Theory — A Volume Dedicated to Gert Sabidussi on
  the Occasion of His 80th Birthday.

\bibitem{ikr2012}
Aleksandar Ili\'c, Sandi Klav\v{z}ar, and Yoomi Rho.
\newblock Generalized lucas cubes.
\newblock {\em Appl. Anal. Discrete Math}, 6:82--94, 04 2012.

\bibitem{survey}
Sandi Klav\v{z}ar.
\newblock Structure of fibonacci cubes: a survey.
\newblock {\em Journal of Combinatorial Optimization}, 25:505--522, 2013.

\bibitem{km2012}
Sandi Klav\v{z}ar and Michel Mollard.
\newblock Cube polynomial of fibonacci and lucas cubes.
\newblock {\em Acta Applicandae Mathematicae}, 117, 02 2012.

\bibitem{klmope2011}
Sandi Klav\v{z}ar, Michel Mollard, and Marko Petkovšek.
\newblock The degree sequence of fibonacci and lucas cubes.
\newblock {\em Discrete Mathematics}, 311(14):1310 -- 1322, 2011.

\bibitem{hsuliu}
J.~{Liu} and W.-J. {Hsu}.
\newblock Distributed algorithms for shortest-path, deadlock-free routing and
  broadcasting in a class of interconnection topologies.
\newblock In {\em Proceedings Sixth International Parallel Processing
  Symposium}, pages 589--596, 1992.

\bibitem{mo2018}
Michel Mollard.
\newblock The existence of perfect codes in a family of generalized fibonacci
  cubes.
\newblock {\em Information Processing Letters}, 140:1 -- 3, 2018.

\bibitem{mupe2001}
Emanuele Munarini, Claudio Perelli~Cippo, and Norma Salvi.
\newblock On the lucas cubes.
\newblock {\em The Fibonacci Quarterly}, 39, 02 2001.

\bibitem{ra2013}
Mark Ramras.
\newblock Congestion-free routing of linear permutations on fibonacci and lucas
  cubes.
\newblock {\em Australasian Journal of Combinatorics}, 60:1--10, 01 2014.

\bibitem{Zag}
Norma Salvi.
\newblock On the existence of cycles of every even length on generalized
  fibonacci cubes.
\newblock {\em Le Matematiche}, 51, 01 2010.

\bibitem{Sol2008}
Faina~I. Solov’eva.
\newblock On perfect binary codes.
\newblock {\em Discrete Applied Mathematics}, 156(9):1488 -- 1498, 2008.
\newblock General Theory of Information Transfer and Combinatorics.

\bibitem{Va1962}
Y.L. Vasil'ev.
\newblock On nongroup close-packed codes.
\newblock {\em Probl. Kibern.}, 8:337--339, 01 1962.

\end{thebibliography}

\end{document}